\documentclass{amsart}

\usepackage{amsmath,amssymb,amsfonts,enumerate,amsthm,graphicx,color}
\usepackage{amscd,amsmath,amsthm,amssymb}
\usepackage{pstcol,pst-plot,pst-3d}
\usepackage{lmodern,pst-node}
\usepackage{lineno}

\def\opn#1#2{\def#1{\operatorname{#2}}} 
\opn\chara{char} \opn\length{\ell} \opn\pd{pd} \opn\rk{rk}
\opn\projdim{proj\,dim} \opn\injdim{inj\,dim} \opn\rank{rank}
\opn\depth{depth} \opn\grade{grade} \opn\height{height}
\opn\embdim{emb\,dim} \opn\codim{codim}

\opn\Tr{Tr} \opn\bigrank{big\,rank}
\opn\superheight{superheight}\opn\lcm{lcm}
\opn\trdeg{tr\,deg}%
\opn\reg{reg} \opn\lreg{lreg} \opn\skel{skel}
\opn\multideg{multideg}
\opn\div{div} \opn\Div{Div} \opn\cl{cl} \opn\Cl{Cl}
\opn\Spec{Spec} \opn\Supp{Supp} \opn\supp{supp} \opn\Sing{Sing}
\opn\Ass{Ass}
\opn\Ann{Ann} \opn\Rad{Rad} \opn\Soc{Soc}
\opn\Ker{Ker} \opn\Coker{Coker} \opn\Im{Im} \opn\Hom{Hom}
\opn\Tor{Tor} \opn\Ext{Ext} \opn\End{End} \opn\Aut{Aut}
\opn\id{id}

\opn\nat{nat}
\opn\pff{pf}
\opn\Pf{Pf} \opn\GL{GL} \opn\SL{SL} \opn\mod{mod} \opn\ord{ord}
\opn\aff{aff} \opn\con{conv} \opn\relint{relint} \opn\st{st}
\opn\lk{lk} \opn\cn{cn} \opn\core{core} \opn\vol{vol}
\opn\link{link} \opn\star{star} \opn\skel{skel} \opn\Reg{Reg}
\opn\gr{gr}

\def\pot#1#2{#1[\kern-0.28ex[#2]\kern-0.28ex]}

\opn\dirlim{\underrightarrow{\lim}}
\opn\inivlim{\underleftarrow{\lim}}
\def\Implies{\ifmmode\Longrightarrow \else
     \unskip${}\Longrightarrow{}$\ignorespaces\fi}
\def\implies{\ifmmode\Rightarrow \else
     \unskip${}\Rightarrow{}$\ignorespaces\fi}
\def\iff{\ifmmode\Longleftrightarrow \else
     \unskip${}\Longleftrightarrow{}$\ignorespaces\fi}

\let\:=\colon

\newtheorem{thm}{Theorem}[section]
\newtheorem{cor}[thm]{Corollary}
\newtheorem{lem}[thm]{Lemma}
\newtheorem{prop}[thm]{Proposition}

\newtheorem{exam}[thm]{Example}
\newtheorem{rem}[thm]{Remark}

\numberwithin{equation}{section}

\begin{document}

\bibliographystyle{amsplain}

\title{ Some Cohen-Macaulay and unmixed binomial edge ideals  }
\author{ Dariush Kiani and Sara Saeedi Madani }

\address{Dariush Kiani, Department of Pure Mathematics,
 Faculty of Mathematics and Computer Science,
 Amirkabir University of Technology (Tehran Polytechnic),
424, Hafez Ave., Tehran 15914, Iran, and School of Mathematics, Institute for Research in Fundamental Sciences (IPM),
P.O. Box 19395-5746, Tehran, Iran.} 
\email{dkiani@aut.ac.ir, dkiani7@gmail.com}
\address{Sara Saeedi Madani, Universit\"at Osnabr\"uck, Institut f\"ur Mathematik, 49069 Osnabr\"uck, Germany.} 
\email{sara.saeedimadani@uni-osnabrueck.de, sarasaeedim@gmail.com}

\begin{abstract}
We study unmixed and Cohen-Macaulay properties of the binomial edge ideal of some classes of graphs. We compute the depth of the binomial edge ideal of a generalized block graph. We also characterize all generalized block graphs whose binomial edge ideals are Cohen-Macaulay and unmixed. So that we generalize the results of Ene, Herzog and Hibi on block graphs. Moreover, we study unmixedness and Cohen-Macaulayness of the binomial edge ideal of some graph products such as the join and corona of two graphs with respect to the original graphs'.
\end{abstract}

\maketitle

\section{ Introduction }\label{Introduction}

\noindent The binomial edge ideal of a graph was introduced, at about the same time, in \cite{HHHKR} and \cite{O}. Let $G$ be a simple graph with the vertex set $[n]$, edge set $E(G)$, and let $S=K[x_1,\ldots,x_n,y_1,\ldots,y_n]$ be the polynomial ring, where $K$ is a field. The \textbf{binomial edge ideal} of $G$ in $S$, which is denoted by $J_G$, is generated by binomials $f_{ij}=x_iy_j-x_jy_i$, where $i<j$ and $\{i,j\}\in E(G)$. Moreover, this ideal is generated by a collection of 2-minors of a $(2\times n)$-matrix whose entries are all indeterminates. Some algebraic properties of this ideal
were studied in \cite{EHH}, \cite{HHHKR}, \cite{KS} and \cite{SK}. Also, in \cite{EHHQ}, the binomial edge ideal of a pair of graphs were introduced as a
generalization of the binomial edge ideal of a graph. Let $G_1$ be a graph on the vertex set $[m]$ and $G_2$ a graph on the vertex set $[n]$, and
let $X= (x_{ij})$ be an $(m\times n)$-matrix of indeterminates. Let
$S=K[X]$ be the polynomial ring in the variables $x_{ij}$, where $i=1,\ldots,m$ and $j=1,\ldots,n$. Let $e=\{i,j\}$ for some $1\leq i < j\leq m$
and $f=\{t, l\}$ for some $1\leq t < l\leq n$. To
the pair $(e,f)$, the $2$-minor $p_{e,f}=[i,j|t,l]=x_{it}x_{jl}-x_{il}x_{jt}$ of $X$ is assigned. The ideal $J_{G_1,G_2}=(p_{e,f}:~e\in E(G_1), f\in E(G_2))$ is called the \textbf{binomial edge ideal of the pair} $(G_1,G_2)$. Some properties of this ideal were also studied in \cite{SK1}. If $G_1$ is just an edge, then $J_{G_1,G_2}$ is isomorphic to $J_{G_2}$.

We study unmixedness and Cohen-Macaulayness of the binomial edge ideal of a class of chordal graphs, called generalized block graphs. Hence, we generalize the result of \cite{EHH} about block graphs. Also, we study the join and corona of two graphs. Actually, we investigate about the relationship between unmixed and Cohen-Macaulay  properties of the binomial ideal associated to the join and corona of two graphs, and the original graphs. This paper is organized as follows. In Section~\ref{Preliminaries}, we review some definitions, notation and known results on the topic, which will be used throughout the paper. In Section~\ref{Binomial edge ideal of a class of chordal graphs}, we study a class of chordal graphs, called generalized block graphs, introduced previously by the authors in \cite{KS}. We compute the depth of the binomial edge ideal of these graphs with respect to graphical terms. Further, we determine all such graphs whose binomial edge ideals are unmixed and Cohen-Macaulay. Therefore, we generalize the results appeared in \cite{EHH}, about block graphs. In Section~\ref{Join}, we focus on the join of graphs. We give some necessary and sufficient conditions for unmixed and Cohen-Macaulay properties of
the binomial edge ideal of the join of two graphs. We divide it to three different cases; for two connected graphs, for one connected graph and a disconnected one, and for two disconnected graphs. Also, we generalize the results of \cite{RR} about a cone over a graph. Finally, in Section~\ref{Corona}, we study some properties of the binomial edge ideal of the corona of two graphs.

In this paper, we mean by a graph, a simple graph. Also, if $V=\{v_1,\ldots,v_n\}$ is the vertex set of $G$, then we denote it by $[n]$. By the ring $S$, we mean the polynomial ring in the variables $x_{ij}$, where $i=1,\ldots,m$ and $j=1,\ldots,n$. If $m=2$, then we consider $S$ to be the ring $K[x_1,\ldots,x_n,y_1,\ldots,y_n]$, for simplicity.

\section{ Preliminaries }\label{Preliminaries}

\noindent In this section, we pose some notions and facts, which we need in the sequel. Suppose that $G$ is a graph on $[n]$. A vertex $v$ of $G$ whose deletion from the graph, implies a graph with more connected components than $G$, is called a \textbf{cut point} of $G$. Let $T$ be a subset of $[n]$, and let $G_1,\ldots,G_{c_G(T)}$ be the connected
components of $G_{[n]\setminus T}$, the induce subgraph of $G$ on $[n]\setminus T$. For each $G_i$, we denote by $\widetilde{G}_i$ the complete graph on the vertex set $V(G_i)$. If there is no confusion, we might write $c(T)$ instead of $c_G(T)$, and set $$P_T(G)=(\bigcup_{i\in T}\{x_i,y_i\}, J_{\widetilde{G}_1},\ldots,J_{\widetilde{G}_{c(T)}}).$$ Then, $P_T(G)$ is a prime ideal, where $\mathrm{height}\hspace{0.35mm}P_T(G)=n+|T|-c(T)$, by \cite[Lemma~3.1]{HHHKR}. Moreover, $J_G=\bigcap_{T\subset [n]}P_T(G)$, by \cite[Theorem~3.2]{HHHKR}. So that, $\mathrm{dim}\hspace{0.35mm}S/J_G=\mathrm{max}\{n-|T|+c(T):T\subset [n]\}$, by \cite[Cororally~3.3]{HHHKR}. If each $i\in T$ is a cut point of the graph $G_{([n]\setminus T)\cup \{i\}}$, then we say that $T$ has \textbf{cut point
property} for $G$. Let $\mathcal{C}(G)=\{\emptyset\}\cup \{T\subset [n]:T~\mathrm{has~cut~point~property~for}~G\}$. One has $\mathcal{C}(G)=\{\emptyset\}$ if
and only if $G$ is a complete graph. Denoted by $\mathcal{M}(G)$, we mean the set of all minimal prime ideals of $J_G$. We use $\overline{\mathcal{C}}(G)$ to denote $\mathcal{C}(G)\setminus \{\emptyset\}$. We have that $T\in \mathcal{C}(G)$ if and only if $P_T(G)\in \mathcal{M}(G)$, by \cite[Corollary~3.9]{HHHKR}.

\begin{prop}\label{unmixed1}
\cite[Lemma~2.5]{RR} Let $G$ be a connected graph. Then the following conditions are equivalent:\\
{\em{(a)}} $J_G$ is unmixed. \\
{\em{(b)}} For all $T\in \mathcal{C}(G)$, we have $c(T)=|T|+1$.
\end{prop}

In \cite{EHHQ}, the authors classified all pairs of graphs $(G_1,G_2)$ such that $J_{G_1,G_2}$ is unmixed:

\begin{prop}\label{unmixed2}
\cite[Proposition~4.1]{EHHQ} Let $n\geq m\geq 3$ be integers and let $G_1$ and $G_2$ be connected simple
graphs with $V(G_1)=[m]$ and $V(G_2)=[n]$. Then the binomial edge ideal $J_{G_1,G_2}$ is
unmixed if and only if $G_1$ is complete and for all $T\in \mathcal{C}(G_2)$, one has $(c(T)-1)(m-1)=|T|$.
\end{prop}

So, by combining these two propositions, one has the following, which will be used in this paper.

\begin{prop}\label{unmixed3}
Let $n\geq m\geq 2$ be integers and let $G_1$ and $G_2$ be connected simple
graphs with $V(G_1)=[m]$ and $V(G_2)=[n]$. Then the binomial edge ideal $J_{G_1,G_2}$ is
unmixed if and only if $G_1$ is complete and for all $T\in \mathcal{C}(G_2)$, one has $(c(T)-1)(m-1)=|T|$.
\end{prop}

\section{ Binomial edge ideal of a class of chordal graphs }\label{Binomial edge ideal of a class of chordal graphs}

\noindent In this section, we study unmixed and Cohen-Macaulay properties of the binomial edge ideal of a class of chordal graphs, called generalized block graphs. Indeed, we generalize a result of Ene, Herzog and Hibi in \cite{EHH}. Generalized block graphs were introduced in \cite{KS}, as a generalization of block graphs. Here, we recall the definition. A \textbf{nonseparable} graph is a connected nontrivial graph which does not have any cut points. A \textbf{block} of a graph is a maximal nonseparable subgraph of it. A \textbf{block graph} is a connected graph whose blocks are cliques. Here, by a clique, we mean a complete subgraph of a graph. One could see that a graph $G$ is a block graph if and only if it is a chordal graph such that every two maximal cliques of it intersect in at most one vertex. This class was considered in~\cite[Theorem~1.1]{EHH}. Recall that $\Delta(G)$ is the clique complex of the graph $G$, the simplicial complex whose facets are the vertex sets of the maximal cliques of $G$. Now, let $G$ be a connected chordal graph such that for every three maximal cliques of $G$ which have a nonempty intersection, the intersection of each pair of them is the same. In other words, $G$ has this property that for every $F_i,F_j,F_k\in \Delta(G)$, if $F_i\cap F_j\cap F_k\neq \emptyset$, then $F_i\cap F_j=F_i\cap F_k=F_j\cap F_k$. We call $G$, a \textbf{generalized block} graph. Thus, it is clear that all block graphs and hence all trees are also generalized block graphs. For example, the graphs depicted in Figure~\ref{generalized2} and Figure~\ref{generalized1} are generalized block graphs which are not block graphs. Here, by a \textbf{cut set} of a graph $G$, we mean a subset of vertices of $G$ whose deletion increases the number of connected components of $G$. Moreover, by a \textbf{minimal cut set} of $G$, we mean a cut set which is minimal under inclusion. One could see that a subset $A$ of the vertices of a generalized block graph $G$ is a minimal cut set if and only if there exist $F_{i_1},\ldots,F_{i_t}\in \Delta(G)$ such that $\bigcap_{j=1}^{t}F_{i_j}=A$, and for all other facets $F$ of $\Delta(G)$, $F\cap A=\emptyset$. In this case, we sometimes say that $A$ is a $t$-minimal cut set of $G$.

Now, recall that the \textbf{clique number} of a graph $G$, denoted by $\omega(G)$, is the maximum size of the maximal cliques of $G$. Let $G$ be a generalized block graph on $[n]$. For each $i=1,\ldots,\omega(G)-1$, we set $$\mathcal{A}_{i}(G):=\{A\subseteq [n] : |A|=i,A~\mathrm{is~a~minimal~cut~set~of~}G\}.$$
Also, we put $a_i(G):=|\mathcal{A}_{i}(G)|$, for all $i=1,\ldots,\omega(G)-1$. Thus, a generalized block graph $G$ is a block graph if and only if $a_i(G)=0$, for all $i>1$.

\begin{exam}\label{generalized block1}
{\em (a) Let $G_1$ be the graph shown in Figure~\ref{generalized2}. Then, it has two $3$-minimal cut sets, $A=\{x,y\}$ and $A'=\{z,w\}$. Thus, ${\mathcal{A}}_2(G_1)=\{A,A'\}$, and hence $a_2(G_1)=2$. But, $a_i(G_1)=0$ for all $i\neq 2$, since $G_1$ has no other minimal cut set. \\
\indent (b) Let $G_2$ be the graph shown in Figure~\ref{generalized1}. Then it has two minimal cut sets; $A=\{x\}$ is a $2$-minimal cut set and $A'=\{y,z\}$ is a $3$-minimal cut set. So, ${\mathcal{A}}_1(G_2)=\{A\}$ and ${\mathcal{A}}_2(G_2)=\{A'\}$, which imply that $a_1(G_2)=1$, $a_2(G_2)=1$ and $a_i(G_2)=0$ for all $i\neq 1,2$.  }
\end{exam}

Recall that a facet $F$ of a simplicial complex $\Delta$ is a \textbf{leaf}, if either $F$ is the only facet, or there exists a facet $G$, called a \textbf{branch} of $F$, such that for each facet $H$ of $\Delta$, with $H\neq F$, one has $H\cap F \subseteq G\cap F$. Each leaf $F$ has at least a free vertex. A simplicial complex $\Delta$ is called a \textbf{quasi-forest}, if its facets can be ordered as $F_1,\ldots,F_r$ such that for all $i>1$, $F_i$ is a leaf of the subcomplex of $\Delta$ with facets $F_1,\ldots,F_{i-1}$. Such an order is called a \textbf{leaf order}. Now, we compute the depth of the binomial edge ideal of a generalized block graph in the following theorem.
\begin{center}
\begin{figure}
\hspace{0 cm}
\includegraphics[height=2.1cm,width=3.2cm]{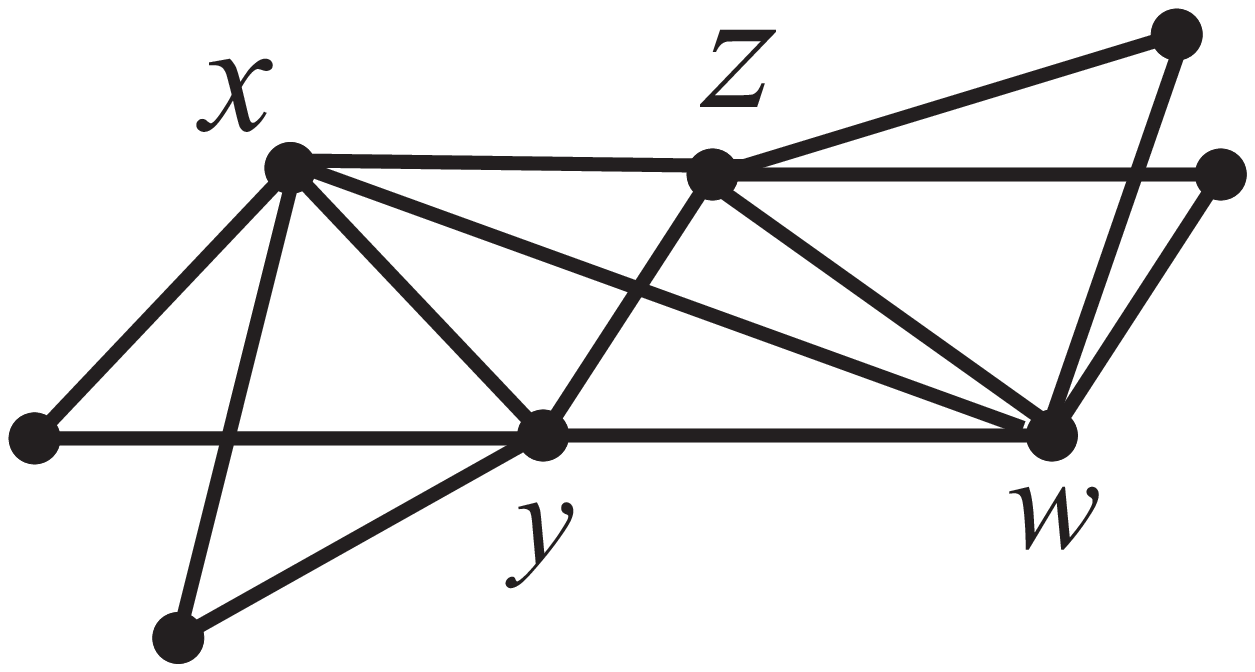}
\caption{\footnotesize{}A generalized block graph}\hspace{2 cm}
\label{generalized2}
\end{figure}
\end{center}
\begin{center}
\begin{figure}
\hspace{0 cm}
\includegraphics[height=1.9cm,width=3.2cm]{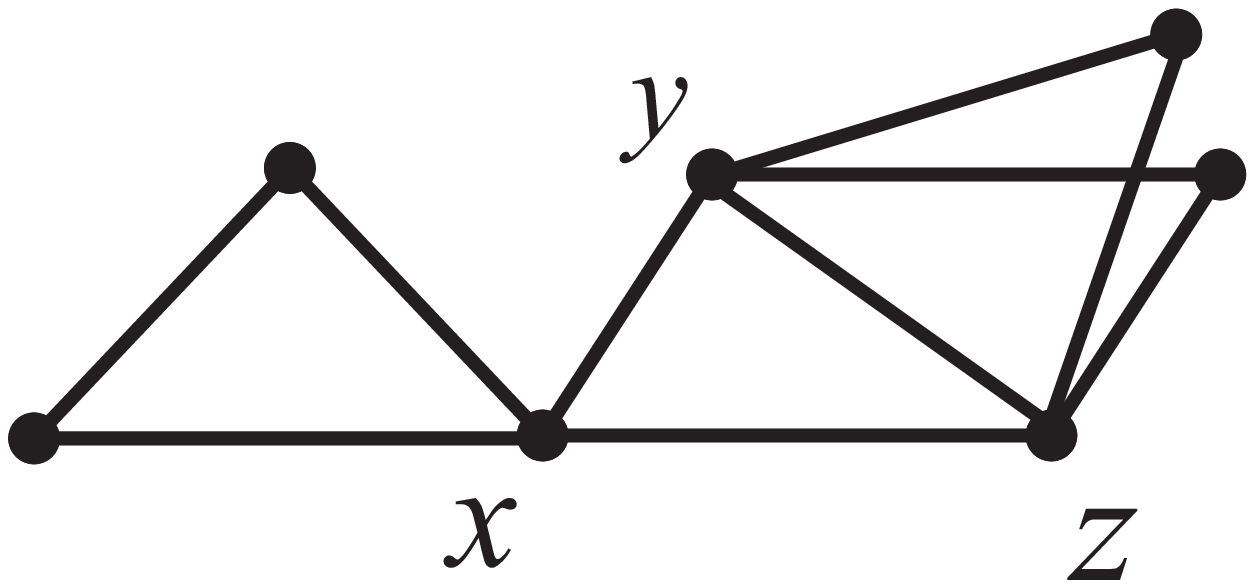}
\caption{\footnotesize{}A generalized block graph}\hspace{2 cm}
\label{generalized1}
\end{figure}
\end{center}

\begin{thm}\label{depth-generalized}
Let $G$ be a generalized block graph on $[n]$, with $r$ connected components. Then we have $$\mathrm{depth}(S/J_G)=n+r-\sum_{i=2}^{\omega(G)-1}(i-1)a_i(G).$$
\end{thm}

\begin{proof}
The proof is based on the technique applied in the proof of \cite[Theorem~1.1]{EHH}, and \cite[Theorem~3.18]{KS}. By Dirac's theorem (see \cite{D}), we have that $\Delta(G)$ is a quasi forest, since $G$ is chordal. Let $F_1,\ldots,F_c$ be a leaf order of the facets of $\Delta(G)$. We use induction on $c=c(G)$, the number of maximal cliques of $G$. If $c=1$, then $G$ is complete and it is well-known that $\mathrm{depth}(S/J_G)=n+1$. Let $c>1$. Note that we can assume that $G$ is a connected graph, because if $G_1,\ldots,G_r$ are the connected components of $G$ on $[n_1],\ldots,[n_r]$, respectively, then $\mathrm{depth}(S/J_G)=\sum_{i=1}^{p}\mathrm{depth}(S_i/J_{G_i})$, where $S_i=K\big{[}x_j,y_j:j\in [n_i]\big{]}$. Let $F_{t_1},\ldots,F_{t_q}$ be all the branches of the leaf $F_c$. Note that $q\geq 1$. Since $G$ is a generalized block graph, each pair of the facets $F_c,F_{t_1},\ldots,F_{t_q}$ intersect in exactly the same set of vertices, like $A$ with $|A|=\alpha\geq 1$, and also, $F_c\cap F_l=\emptyset$, for all $l\neq t_1,\ldots,t_q$, as $F_c$ is a leaf. Then, for all $l\neq t_1,\ldots,t_q$, we have $A\cap F_l=\emptyset$. Actually, $A$ is a $(q+1)$-minimal cut set of $G$. Now, let $J_G=Q\cap Q'$, where
\begin{equation}
Q=\bigcap_{\substack{
T\subseteq [n] \\
A\nsubseteq T
}}
P_T(G)~~,~~Q'=\bigcap_{\substack{
T\subseteq [n] \\
A\subseteq T
}}
P_T(G).
\nonumber
\end{equation}
Note that we have
\begin{equation}
Q=\bigcap_{\substack{
T\in \mathcal{C}(G) \\
A\nsubseteq T
}}
P_T(G)=
\bigcap_{\substack{
T\in \mathcal{C}(G) \\
A\cap T=\emptyset
}}
P_T(G),
\nonumber
\end{equation}
since similar to the proof of \cite[Theorem~3.18]{KS}, for $T\subset [n]$ with $T\in \mathcal{C}(G)$, we have that $A\nsubseteq T$ if and only if $A\cap T=\emptyset$; as if $|A|>1$ and $v\in A\cap T$, then $v$ is not a cut point of the graph $G_{([n]\setminus T)\cup \{v\}}$, since $A\setminus T\neq \emptyset$, so we get a contradiction.

Let $G'$ be the graph obtained from $G$, by replacing the cliques $F_c,F_{t_1},\ldots,F_{t_q}$, by the clique on the vertex set $F_c\cup (\bigcup_{j=1}^qF_{t_j})$. One could see that $Q=J_{G'}$ and $Q'=(x_i,y_i:i\in A)+J_{G_{[n]\setminus A}}$. Thus, $Q+Q'=(x_i,y_i:i\in A)+J_{{G'}_{[n]\setminus A}}$. Therefore, we have $S/Q=S/J_{G'}$, $S/Q'\cong S_A/J_{G_{[n]\setminus A}}$ and $S/(Q+Q')\cong S_A/J_{{G'}_{[n]\setminus A}}$, where $S_A=K[x_i,y_i:i\in [n]\setminus A]$. It is not difficult to observe that the graphs $G'$, $G_{[n]\setminus A}$ and ${G'}_{[n]\setminus A}$ are generalized block graphs. Note that by the construction of generalized block graphs, we have for all $i\neq \alpha$, ${\mathcal{A}}_i(G')={\mathcal{A}}_i({G'}_{[n]\setminus A})={\mathcal{A}}_i(G)$ and ${\mathcal{A}}_i(G_{[n]\setminus A})\subseteq {\mathcal{A}}_i(G)$. But, ${\mathcal{A}}_{\alpha}(G')={\mathcal{A}}_{\alpha}({G'}_{[n]\setminus A})={\mathcal{A}}_{\alpha}(G)\setminus \{A\}$ and ${\mathcal{A}}_{\alpha}(G_{[n]\setminus A})\subseteq {\mathcal{A}}_{\alpha}(G)\setminus \{A\}$. Thus, we have for all $i\neq \alpha$, $a_i(G')=a_i({G'}_{[n]\setminus A})=a_i(G)$, $a_i(G_{[n]\setminus A})\leq a_i(G)$, also $a_{\alpha}(G')=a_{\alpha}({G'}_{[n]\setminus A})=a_{\alpha}(G)-1$ and $a_{\alpha}(G_{[n]\setminus A})\leq a_{\alpha}(G)-1$. Since the number of maximal cliques of $G'$ is less than $G$, by induction hypothesis, we get
\begin{eqnarray}
\mathrm{depth}(S/J_{G'})&=& n+1-\sum_{i=2}^{\omega(G')-1}(i-1)a_i(G')
\nonumber\\
\nonumber &= & {}
n+1-\sum_{\substack{
i=2 \\
i\neq \alpha
}}^{\omega(G)-1}(i-1)a_i(G)-(\alpha-1)(a_{\alpha}(G)-1)
\nonumber\\
\nonumber &=& {} n+\alpha-\sum_{i=2}^{\omega(G)-1}(i-1)a_i(G).
\nonumber
\end{eqnarray}
On the other hand, $G_{[n]\setminus A}$ has $q+1$ connected components, say $H_1,\ldots,H_{q+1}$ on $[n_1],\ldots,[n_{q+1}]$, respectively. So, the induction hypothesis is valid for these connected components, and hence we have
\begin{eqnarray}
\mathrm{depth}(S_A/J_{G_{[n]\setminus A}})&=& \sum_{j=1}^{q+1}\mathrm{depth}(S_A^j/J_{H_j})
\nonumber\\
\nonumber &= & {}
\sum_{j=1}^{q+1}\big{(}n_j+1-\sum_{i=2}^{\omega(H_j)-1}(i-1)a_i(H_j)\big{)}
\nonumber\\
\nonumber &=& {} (n-\alpha)+(q+1)-\sum_{i=2}^{\omega(G_{[n]\setminus A})-1}(i-1)a_i(G_{[n]\setminus A})
\nonumber\\
\nonumber &\geq & {} (n-\alpha)+(q+1)-\sum_{\substack{
i=2 \\
i\neq \alpha
}}^{\omega(G)-1}(i-1)a_i(G)-(\alpha-1)(a_{\alpha}(G)-1)
\nonumber\\
\nonumber &=& {} n+q-\sum_{i=2}^{\omega(G)-1}(i-1)a_i(G),
\nonumber
\end{eqnarray}
where $S_A^j=K\big{[}x_i,y_i:i\in [n_j]\big{]}$, for $j=1,\ldots,q+1$.
Moreover, since the number of maximal cliques of ${G'}_{[n]\setminus A}$ is less than $G$, by the induction hypothesis, we get
\begin{eqnarray}
\mathrm{depth}(S/J_{{G'}_{[n]\setminus A}})&=& (n-\alpha)+1-\sum_{i=2}^{\omega({G'}_{[n]\setminus A})-1}(i-1)a_i({G'}_{[n]\setminus A})
\nonumber\\
\nonumber &= & {}
(n-\alpha)+1-\sum_{\substack{
i=2 \\
i\neq \alpha
}}^{\omega(G)-1}(i-1)a_i(G)-(\alpha-1)(a_{\alpha}(G)-1)
\nonumber\\
\nonumber &=& {} n-\sum_{i=2}^{\omega(G)-1}(i-1)a_i(G).
\nonumber
\end{eqnarray}
Now, it is enough to apply the depth lemma (see for example \cite[Proposition~1.2.9]{BH}) to the short exact sequence $$0\rightarrow S/J_G\rightarrow S/Q\oplus S/Q'\rightarrow S/(Q+Q')\rightarrow 0.$$ Then, we get the desired conclusion.
\end{proof}

Consequently, we show that the only generalized block graphs whose binomial edge ideals are Cohen-Macaulay, are those block graphs discussed in \cite[Theorem~1.1]{EHH}:

\begin{cor}\label{CM-generalized}
Let $G$ be a generalized block graph. Then $S/J_G$ is Cohen-Macaulay if and only if $G$ is a block graph whose each vertex is the intersection of at most two maximal cliques.
\end{cor}

\begin{proof}
If $S/J_G$ is Cohen-Macaulay, then by \cite[Corollary~3.4]{HHHKR}, we have $\mathrm{dim}(S/J_G)=n+r$. So that $\mathrm{depth}(S/J_G)=n+r$. Thus, by Theorem~\ref{depth-generalized}, we have $\sum_{i=2}^{\omega(G)-1}(i-1)a_i(G)=0$, which implies that $a_i(G)=0$, for all $i>1$. Hence, $G$ is a block graph. Then, by using \cite[Theorem~1.1]{EHH}, we get the result.
\end{proof}

While unmixedness and Cohen-Macaulayness of the binomial edge ideal are equivalent for block graphs, this is not true for generalized block graphs, as it is shown in the next theorem. Since it is known that the binomial edge ideal of a graph is unmixed if and only if the binomial edge ideal of its connected components are unmixed, here we focus on the connected case:

\begin{thm}\label{unmixed-generalized}
Let $G$ be a connected generalized block graph on $[n]$. Then $J_G$ is unmixed if and only if the following conditions hold: \\
{\em{(a)}} For every $t$-minimal cut set $A$ of $G$, we have $|A|=t-1$; \\
{\em{(b)}} If there are some minimal cut sets $A_1,\ldots,A_s$ whose union is a maximal clique of $G$, then $|A_i|=1$, for some $i=1,\ldots,s$.
\end{thm}

\begin{proof}
Suppose that $J_G$ is unmixed. Let $A$ be a $t$-minimal cut set of $G$. Then, clearly, $A\in \mathcal{C}(G)$ and $c_G(A)=t$. Thus, by Proposition~\ref{unmixed1}, we have $|A|=c_G(A)-1=t-1$ and hence condition (a) holds. Now, for all $i=1,\ldots,s$, let $A_i$ be a $t_i$-minimal cut set of $G$ whose union is a maximal clique of $G$ such that $|A_i|>1$. Then, by (a), we have $t_i>2$, so that $A_i$ is the intersection of at least three maximal cliques of $G$. Thus, by the construction of generalized block graphs, $T:=\bigcup_{i=1}^sA_i \in \mathcal{C}(G)$. But, it is easy to see that $c_G(T)=|T|=\sum_{i=1}^st_{i}-s$, which contradicts to the unmixedness of $J_G$, by Proposition~\ref{unmixed1}. Thus, condition (b) also holds. For the converse, let $\emptyset \neq T\in \mathcal{C}(G)$. One could see that $T=\bigcup_{j=1}^pA_{i_j}$, where $A_{i_j}$ is a $t_{i_j}$-minimal cut set of $G$. By (a), we have $|T|=\sum_{j=1}^p|A_{i_j}|=\sum_{j=1}^pt_{i_j}-p$. Note that if there are some of these $A_{i_j}$'s whose union is a maximal clique of $G$, then one of them should be a singleton, say $\{v\}$, by (b). But, in this case $T\notin \mathcal{C}(G)$, since $v$ is the intersection of exactly two maximal cliques of $G$, by (a), and hence is not a cut point of $G_{([n]\setminus T)\cup \{v\}}$. So, we get a contradiction, and hence $T$ does not contain any maximal cliques of $G$. Thus, there is a vertex of each maximal clique of $G$ in the graph $G_{[n]\setminus T}$, so that  $c_G(T)=\sum_{j=1}^pt_{i_j}-p+1$. Therefore, applying Proposition~\ref{unmixed1}, we obtain that $J_G$ is unmixed, as desired.
\end{proof}

\begin{exam}\label{generalized block2}
{\em We use the notation of Example~\ref{generalized block1}. The ideal $J_{G_1}$ is not unmixed, by the above theorem. While condition (a) of the theorem is true for this graph, but $A\cup A'$ is the set of the vertices of a maximal clique of $G_1$, which is isomorphic to $K_4$. So that condition (b) does not hold for this graph, as $A$ and $A'$ are not singletons. On the other hand, by Theorem~\ref{unmixed-generalized}, the ideal $J_{G_2}$ is unmixed, but $S/J_{G_2}$ is not Cohen-Macaulay, since it is not a block graph. }
\end{exam}

\section{ Binomial edge ideal of the join of graphs }\label{Join}

\noindent In this section, we investigate about unmixedness and Cohen-Macaulayness of binomial ideals associated to the join of connected and disconnected graphs. Let $G$ and $H$ be two graphs on $[m]$ and $[n]$, respectively. We denote by $G*H$, the \textit{join} (product) of two graphs $G$ and $H$, that is
the graph with vertex set $[m]\cup [n]$, and the edge set $E(G)\cup E(H)\cup \{\{v,w\}~:~v\in [m],~w\in [n]\}$. In particular, the cone of a vertex $v$ on a graph $G$ is defined to be their join, that is $v*G$, and is denoted by $\mathrm{cone}(v,G)$. Let $V$ be a set. To simplify our notation throughout this paper, we introduce the join of two collection of subsets of $V$, $\mathcal{A}$ and $\mathcal{B}$, denoted by $\mathcal{A}\circ \mathcal{B}$, as $\{A\cup B: A\in \mathcal{A}, B\in \mathcal{B}\}$. If $\mathcal{A}_1,\ldots,\mathcal{A}_t$ are collections of subsets of $V$, then we denote their join, by $\bigcirc_{i=1}^{t}\mathcal{A}_i$. If $\mathcal{A}$ is empty, then $\mathcal{A}\circ \mathcal{B}=\emptyset$, for every $\mathcal{B}$.

\subsection{Join of two connected graphs}\label{Join of two connected graphs}

The following proposition determines all minimal prime ideals of the binomial edge ideal of the join of two connected graphs $G_1$ and $G_2$ with respect to those of $G_1$ and $G_2$.

\begin{prop}\label{both connected}
If $G_1$ and $G_2$ are connected graphs on disjoint sets of vertices $[n_1]$ and $[n_2]$, respectively, then we have \\\\
{\em{(a)}} $\mathcal{C}(G_1*G_2)=\{\emptyset\}\cup \big{(}\overline{\mathcal{C}}(G_1)\circ \{[n_2]\}\big{)}\cup \big{(}\overline{\mathcal{C}}(G_2)\circ \{[n_1]\}\big{)}$. \\\\
{\em{(b)}} $\mathrm{height}\hspace{0.35mm}J_{G_1*G_2}=\mathrm{min}\{\mathrm{height}\hspace{0.35mm}P_{T_1}(G_1)+2n_2,
\mathrm{height}\hspace{0.35mm}P_{T_2}(G_2)+2n_1,n_1+n_2-1:T_1\in \overline{\mathcal{C}}(G_1), T_2\in \overline{\mathcal{C}}(G_2)\}$. \\\\
{\em{(c)}} $\mathrm{dim}\hspace{0.35mm}S/J_{G_1*G_2}=\mathrm{max}\{\mathrm{dim}\hspace{0.35mm}S_1/P_{T_1}(G_1),\mathrm{dim}\hspace{0.35mm}S_2/P_{T_2}(G_2),n_1+n_2+1: T_1\in \overline{\mathcal{C}}(G_1), T_2\in \overline{\mathcal{C}}(G_2)\}$, where $S_1=K\big{[}x_i,y_i:i\in [n_1]\big{]}$ and $S_2=K\big{[}x_i,y_i:i\in [n_2]\big{]}$.
\end{prop}

\begin{proof}
Suppose that $G:=G_1*G_2$ and $n=n_1+n_2$.\\
\indent (a) Let $T\in \overline{\mathcal{C}}(G_1)\circ \{[n_2]\}$. So, $T=T_1\cup [n_2]$, where $\emptyset \neq T_1\in \mathcal{C}(G_1)$. We show that $T$ has cut point property, and hence $T\in \mathcal{C}(G)$. Let $i\in T$. If $i\in T_1$, then $G_{([n]\setminus T)\cup \{i\}}={G_1}_{([n_1]\setminus T_1)\cup \{i\}}$. In this case, $i$ is a cut point of ${G_1}_{([n_1]\setminus T_1)\cup \{i\}}$, since $T_1\in \mathcal{C}(G_1)$. So that $i$ is also a cut point of $G_{([n]\setminus T)\cup \{i\}}$. If $i\in [n_2]$, then $G_{([n]\setminus T)\cup \{i\}}=i*{G_1}_{([n_1]\setminus T_1)}$. So, $i$ is a cut point of $G_{([n]\setminus T)\cup \{i\}}$, since $T_1\in \mathcal{C}(G_1)$ and hence $G_{([n_1]\setminus T_1)}$ is disconnected. Thus, in both cases, $T$ has cut point property. Similarly, the elements of $\overline{\mathcal{C}}(G_2)\circ \{[n_1]\}$ are contained in $\mathcal{C}(G)$. For the other inclusion, let $\emptyset \neq T\in \mathcal{C}(G)$. If $T$ does not contain $[n_1]$ and $[n_2]$, then $G_{[n]\setminus T}$ is connected, and hence no element $i$ of $T$ is a cut point of $G_{([n]\setminus T)\cup \{i\}}$. So, we have $[n_1]\subseteq T$ or $[n_2]\subseteq T$. On the other hand, because $G_1$ and $G_2$ are connected, we have $T\neq [n_1]$ or $T\neq [n_2]$. If $[n_1]\subsetneq T$, then
there is $\emptyset \neq T_2\subsetneq [n_2]$, such that $T=[n_1]\cup T_2$. We have $G_{[n]\setminus T}={G_2}_{[n_2]\setminus T_2}$. Since $T\in \mathcal{C}(G)$, each element $i$ of $T_2$ is a cut point of $G_{([n]\setminus T)\cup \{i\}}$, so that it is also a cut point of ${G_2}_{([n_2]\setminus T_2)\cup \{i\}}$. Hence, $T_2\in \mathcal{C}(G_2)$, and so $T\in \mathcal{C}(G_2)\circ \{[n_1]\}$. If $[n_2]\subsetneq T$, then similarly we get $T\in \mathcal{C}(G_1)\circ \{[n_2]\}$.\\
\indent (b) Let $\emptyset \neq T\in \mathcal{C}(G)$. By (a), $T\in \overline{\mathcal{C}}(G_1)\circ \{[n_2]\}$ or $T\in \overline{\mathcal{C}}(G_2)\circ \{[n_1]\}$. If $T\in \overline{\mathcal{C}}(G_1)\circ \{[n_2]\}$, then we have $|T|=|T_1|+n_2$ and $c_G(T)=c_{G_1}(T_1)$, for some $\emptyset \neq T_1\in \mathcal{C}(G_1)$. Also, we have $\mathrm{height}\hspace{0.35mm}P_{T_1}(G_1)=n_1+|T_1|-c_{G_1}(T_1)$. Hence, $\mathrm{height}\hspace{0.35mm}P_{T}(G)=n+|T|-c_G(T)=2n_2+\mathrm{height}\hspace{0.35mm}P_{T_1}(G_1)$. Similarly, if $T\in \overline{\mathcal{C}}(G_2)\circ \{[n_1]\}$, then
$\mathrm{height}\hspace{0.35mm}P_{T}(G)=2n_1+\mathrm{height}\hspace{0.35mm}P_{T_2}(G_2)$, for some $\emptyset \neq T_2\in \mathcal{C}(G_2)$. Thus, the desired result follows. \\
\indent (c) follows by $\mathrm{dim}\hspace{0.35mm}S/J_{G}=2n-\mathrm{height}\hspace{0.35mm}J_{G}$.
\end{proof}

In the above proposition, if $n_1=1$ or $n_2=1$, then we have the cone of a vertex on a connected graph, which yields
\cite[Lemma~3.1 and Corollary~3.2]{RR}.

\begin{thm}\label{both connected-unmixed}
Let $G_1$, $H_1$ and $H_2$ be connected graphs on $[m]$, $[n_1]$, and $[n_2]$, respectively, where $2\leq m\leq n_1+n_2$. Suppose that $G_2:=H_1*H_2$, and $J_{G_1,H_1}$ and $J_{G_1,H_2}$ are unmixed. Then the following conditions are equivalent:\\
{\em{(a)}} $S/J_{G_1,G_2}$ is Cohen-Macaulay. \\
{\em{(b)}} $J_{G_1,G_2}$ is unmixed. \\
{\em{(c)}} $G_1$ and $G_2$ are complete.
\end{thm}

\begin{proof}
(a) $\Rightarrow$ (b) and (c) $\Rightarrow$ (a) are well-known. \\
\indent (b) $\Rightarrow$ (c) Suppose that $J_{G_1,G_2}$ is unmixed. So, by Proposition~\ref{unmixed3}, $G_1$ is a complete graph. Suppose on the contrary that
$G_2$ is not complete. So, either $H_1$ or $H_2$ is not complete. Without loss of generality, we may assume that $H_1$ is not complete. So, $n_1\geq 2$, and there exists
$\emptyset \neq T_1\in \mathcal{C}(H_1)$. Set $T:=T_1\cup [n_2]$. By Proposition~\ref{both connected}, $T\in \mathcal{C}(G_2)$. Also, we have
$|T|=|T_1|+n_2$ and $c_{G_2}(T)=c_{H_1}(T_1)$. Moreover, by Proposition~\ref{unmixed3}, we have $(c_{G_2}(T)-1)(m-1)=|T|$.
On the other hand, by the assumption, $J_{G_1,H_1}$ is unmixed. Since $G_1$ is complete and $H_1$ is not, we have $m\leq n_1$ and $(c_{H_1}(T_1)-1)(m-1)=|T_1|$, again by Proposition~\ref{unmixed3}. Therefore, we have $|T|=|T_2|$, which implies that $n_2=0$, a contradiction.
\end{proof}

Setting $m=2$ in Theorem~\ref{both connected-unmixed}, we get:

\begin{cor}\label{both connected-m=2}
Let $H_1$ and $H_2$ be connected graphs on $[n_1]$, and $[n_2]$, respectively. Suppose that $J_{H_1}$ and $J_{H_2}$ are unmixed. Then $J_{H_1*H_2}$ is unmixed if and only if $H_1$ and $H_2$ are complete.
\end{cor}

In the above corollary, if $n_1=1$ or $n_2=1$, then we obtain \cite[Theorem~3.3]{RR}:

\begin{cor}\label{both connected-cone}
Let $H$ be a connected graph on $[n]$, and $G=\mathrm{cone}(v,H)$. Assume that $J_{H}$ is unmixed. Then $J_{G}$ is unmixed if and only if $H$ is complete.
\end{cor}

\subsection{Join of a connected graph and a disconnected graph}\label{Join of a connected graph with a disconnected one}

The following determines all minimal prime ideals of the binomial edge ideal of the join of a connected graph $G_1$ and a disconnected graph $G_2$ with respect to those of $G_1$ and $G_2$. If $H$ is a graph with connected components $H_1,\ldots,H_r$, then we denote it by $\bigsqcup_{i=1}^r H_i$.

\begin{prop}\label{one connected}
Suppose that $G_1,H_1,\ldots,H_r$ are connected graphs on disjoint sets of vertices $[t],[n_1],\ldots,[n_r]$, respectively, where $r\geq 2$. If $G_2:=\bigsqcup_{i=1}^r H_i$, and $n:=t+\sum_{i=1}^r n_i$, then we have \\\\
{\em{(a)}} $\mathcal{C}(G_1*G_2)=\{\emptyset\}\cup \big{(}(\bigcirc_{i=1}^r \mathcal{C}(H_i))\circ \{[t]\}\big{)}\cup \big{(}\overline{\mathcal{C}}(G_1)\circ (\{\bigcup_{i=1}^r[n_i]\})\big{)}$. \\\\
{\em{(b)}} $\mathrm{height}\hspace{0.35mm}J_{G_1*G_2}=\mathrm{min}\{2t+\sum_{i=1}^r \mathrm{height}\hspace{0.35mm}J_{H_i},2(n-t)+\mathrm{height}\hspace{0.35mm}P_{T}(G_1),n-1: T\in \overline{\mathcal{C}}(G_1)\}$. \\\\
{\em{(c)}} $\mathrm{dim}\hspace{0.35mm}S/J_{G_1*G_2}=\mathrm{max}\{\sum_{i=1}^r \mathrm{dim}\hspace{0.35mm}S_i/J_{H_i},\mathrm{dim}\hspace{0.35mm}S'/P_{T}(G_1),n+1: T\in \overline{\mathcal{C}}(G_1)\}$, where $S_i=K\big{[}x_j,y_j:j\in [n_i]\big{]}$, for $i=1,\ldots,r$, and $S'=K\big{[}x_j,y_j:j\in [t]\big{]}$.
\end{prop}

\begin{proof}
(a) Set $G:=G_1*G_2$. Let $T\in (\bigcirc_{i=1}^r \mathcal{C}(H_i))\circ [t]$. So, $T=[t]\cup (\bigcup_{i=1}^r T_i)$, where $T_i\in \mathcal{C}(H_i)$ for $i=1,\ldots,r$. We show that $T$ has cut point property, and hence
$T\in \mathcal{C}(G)$. Let $j\in T$. If $j\in T_i$, for some $i=1,\ldots,r$, then $G_{([n]\setminus T)\cup \{j\}}={H_i}_{([n_i]\setminus T_i)\cup \{j\}}\sqcup (\bigsqcup_{l=1,l\neq i}^r {H_l}_{([n_l]\setminus T_l)})$. In this case, $j$ is a cut point of ${H_i}_{([n_i]\setminus T_i)\cup \{j\}}$, since $T_i\in \mathcal{C}(H_i)$. So that $j$ is also a cut point of $G_{([n]\setminus T)\cup \{j\}}$. If $j\in [t]$, then $G_{([n]\setminus T)\cup \{j\}}=j*\bigsqcup_{i=1}^r {H_i}_{([n_i]\setminus T_i)}$. So, $j$ is a cut point of $G_{([n]\setminus T)\cup \{j\}}$, since $G_{([n]\setminus T)}$ is disconnected. Thus, in both cases, $T$ has cut point property. Now, let $T\in \overline{\mathcal{C}}(G_1)\circ (\{\cup_{i=1}^r[n_i]\})$. So, $T=T_0\cup (\bigcup_{i=1}^r[n_i])$, where $\emptyset\neq T_0\in \mathcal{C}(G_1)$. By a similar discussion as above, one can easily see that $T$ has cut point property. For the other inclusion, let $\emptyset \neq T\in \mathcal{C}(G)$, where $T\notin \overline{\mathcal{C}}(G_1)\circ (\{\cup_{i=1}^r[n_i]\})$. If $T$ does not contain $[t]$, then $G_{[n]\setminus T}$ is connected, and hence no element $i$ of $T$ is a cut point of $G_{([n]\setminus T)\cup \{i\}}$. So, we have $[t]\subseteq T$. Let $T:=[t]\cup (\bigcup_{i=1}^r T_i)$, where
$T_i\subseteq [n_i]$, for $i=1,\ldots,r$. Let $1\leq i\leq r$. If $T_i=\emptyset$, then, obviously, $T_i\in \mathcal{C}(H_i)$. If $T_i\neq \emptyset$, then
each $j\in T_i$, is a cut point of $G_{([n]\setminus T)\cup \{j\}}$, since $T\in \mathcal{C}(G)$. So that $j$ is a cut point of ${H_i}_{([n_i]\setminus T_i)\cup \{j\}}$, because $j\in T_i$ and $H_i$'s are on disjoint sets of vertices. Thus, $T_i\in \mathcal{C}(H_i)$. Therefore, $T\in (\bigcirc_{i=1}^r \mathcal{C}(H_i))\circ [t]$. \\
\indent Parts (b) and (c) follow by similar discussions as in the proof of Proposition~\ref{both connected}, by using part~(a).
\end{proof}

\begin{rem}\label{G_1 complete}
{\em With the notation of Proposition~\ref{one connected}, if $G_1$ is the complete graph on $t$ vertices, then $\mathcal{C}(G_1)=\{\emptyset\}$. Thus, by the above proposition, we have $\mathcal{C}(G_1*G_2)=\{\emptyset\}\cup \big{(}(\bigcirc_{i=1}^r \mathcal{C}(H_i))\circ \{[t]\}\big{)}$, so that we get $\mathrm{height}\hspace{0.35mm}J_{G_1*G_2}=\mathrm{min}\{2t+\sum_{i=1}^r \mathrm{height}\hspace{0.35mm}J_{H_i},n-1\}$ and also $\mathrm{dim}\hspace{0.35mm}S/J_{G_1*G_2}=\mathrm{max}\{\sum_{i=1}^r \mathrm{dim}\hspace{0.35mm}S_i/J_{H_i},n+1\}$. But, even in this case, there are some examples in which either of the terms, appeared in the latter formula for the dimension, might be the maximum. For example, let $G_1=H_1=K_2$. Then $\mathrm{dim}\hspace{0.35mm}S/J_{H_1}=3$. If $H_2=K_3$, then $\mathrm{dim}\hspace{0.35mm}S/J_{H_2}=4$, and hence $\mathrm{dim}\hspace{0.35mm}S/J_{G_1*G_2}=n+1=8$, while $\mathrm{dim}\hspace{0.35mm}S_1/J_{H_1}+\mathrm{dim}\hspace{0.35mm}S_2/J_{H_2}=7$. If $H_2=K_{1,4}$, then one can easily check that $\mathrm{dim}\hspace{0.35mm}S/J_{H_2}=8$, and hence $\mathrm{dim}\hspace{0.35mm}S/J_{G_1*G_2}=\mathrm{dim}\hspace{0.35mm}S_1/J_{H_1}+\mathrm{dim}\hspace{0.35mm}S_2/J_{H_2}=11$, while $n+1=10$. }
\end{rem}

In the above proposition, if $t=1$ and $r=2$, then we get \cite[Lemma~3.5 and Corollary~3.6]{RR}.

\begin{thm}\label{one connected-unmixed2}
Let $G_1,H,H_1,\ldots,H_r$ be connected graphs on $[m],[t],[n_1],\ldots,[n_r]$, respectively,
where $r\geq 2$ and $2\leq m\leq t+\sum_{i=1}^r n_i$. Suppose that $G_2:=H*(\bigsqcup_{i=1}^r H_i)$, and $n_i=1$ or $n_i\geq m$, for $i=1,\ldots,r$.
Then $J_{G_1,G_2}$ is unmixed if and only if the following conditions hold:\\
{\em{(a)}} $J_{G_1,H_i}$ is unmixed for all $i=1,\ldots,r$, \\
{\em{(b)}} $G_1$ is complete,  \\
{\em{(c)}} $t=(r-1)(m-1)$, and \\
{\em{(d)}} for every $\emptyset\neq T\in \mathcal{C}(H)$, $(c_H(T)-1)(m-1)=|T|+\sum_{i=1}^r n_i$.
\end{thm}

\begin{proof}
Suppose that $J_{G_1,G_2}$ is unmixed. So, $G_1$ is complete, by Proposition~\ref{unmixed3}. Let $T_1:=[t]$. By Proposition~\ref{one connected}, $T_1\in \mathcal{C}(G_2)$. Since $J_{G_1,G_2}$ is  unmixed, we have $(c_{G_2}(T_1)-1)(m-1)=|T_1|$, by Proposition~\ref{unmixed3}. Thus, $(r-1)(m-1)=t$, because $|T_1|=t$ and $c_{G_2}(T_1)=r$. Let $\emptyset\neq T\in \mathcal{C}(H)$. Then, $T':=T\cup (\bigcup_{i=1}^r[n_i])\in \mathcal{C}(G_2)$, where $|T'|=|T|+\sum_{i=1}^r n_i$ and $c_{G_2}(T')=c_{H}(T)$. So that by Proposition~\ref{unmixed3}, part (d) also follows, since $J_{G_1,G_2}$ is unmixed. Now, let $1\leq i\leq r$. If $n_i=1$, then $J_{G_1,H_i}=(0)$. Suppose that $m\leq n_i$. We show that $J_{G_1,H_i}$ is unmixed. Let
$T_i\in \mathcal{C}(H_i)$ and $T:=T_i\cup [t]$. Then $T\in \mathcal{C}(G_2)$, and we have $|T|=|T_i|+t$ and $c_{G_2}(T)=c_{H_i}(T_i)+r-1$. So, we have
$$(c_{H_i}(T_i)-1)(m-1)=(c_{G_2}(T)-1)(m-1)-(r-1)(m-1)=|T|-t=|T_i|,$$
by unmixedness of $J_{G_1,G_2}$ and part (c). Thus, Proposition~\ref{unmixed3} implies that $J_{G_1,H_i}$ is unmixed. \\
\indent Conversely, let $\emptyset\neq T\in \mathcal{C}(G_2)$. By Proposition~\ref{one connected}, we should consider two cases for $T$. If $T=T'\cup (\bigcup_{i=1}^r[n_i])$, where $\emptyset\neq T'\in \mathcal{C}(H)$, then $|T|=|T'|+\sum_{i=1}^r n_i$ and $c_{G_2}(T)=c_{H}(T')$. So that by part (d), we have $(c_{G_2}(T)-1)(m-1)=|T|$. If $T=[t]\cup (\bigcup_{i=1}^r T_i)$, where $T_i\in \mathcal{C}(H_i)$, for $i=1,\ldots,r$, then $|T|=t+\sum_{i=1}^{r}|T_i|$ and $c_{G_2}(T)=\sum_{i=1}^{r}c_{H_i}(T_i)$. On the other hand, we have $(c_{H_i}(T_i)-1)(m-1)=|T_i|$, for~$i=1,\ldots,r$, because $J_{G_1,H_i}$ is unmixed for all $i=1,\ldots,r$. Thus, we get
$$(c_{G_2}(T)-1)(m-1)=\sum_{i=1}^{r}(c_{H_i}(T_i)-1)(m-1)+(r-1)(m-1)=\sum_{i=1}^{r}|T_i|+t=|T|.$$
These two cases, together with completeness of $G_1$, imply that $J_{G_1,G_2}$ is unmixed, by Proposition~\ref{unmixed3}.
\end{proof}

\begin{rem}\label{omitting part (d)}
{\em Note that in Theorem~\ref{one connected-unmixed2}, when $H$ is a complete graph, then the condition mentioned in part (d) is omitted, since $\mathcal{C}(H)=\{\emptyset\}$. }
\end{rem}

\begin{rem}\label{n_i=1}
{\em In Theorem~\ref{one connected-unmixed2}, when there exists some $n_i\geq m$, the assumption of completeness of $G_1$ in (b), could be omitted. For $J_{G_1,H_i}$ is unmixed and hence $G_1$ is complete, by Proposition~\ref{unmixed3}. But, when for each $i=1,\ldots,r$, $n_i=1$, it is necessary to assume that $G_1$ is complete, since otherwise $J_{G_1,G_2}$ might not be unmixed. }
\end{rem}

\begin{cor}\label{one connected-m=2}
Let $H,H_1,\ldots,H_r$ be connected graphs on $[t],[n_1],\ldots,[n_r]$, respectively,
where $r\geq 2$, and assume that $G:=H*(\bigsqcup_{i=1}^r H_i)$.
Then $J_G$ is unmixed if and only if $H$ is complete, $r=t+1$ and $J_{H_i}$ is unmixed for $i=1,\ldots,r$.
\end{cor}

\begin{proof}
Put $m=2$ in Theorem~\ref{one connected-unmixed2}. Now, it is enough to note that for every $\emptyset\neq T\in \mathcal{C}(H)$, $c_H(T)=|T|+1+\sum_{i=1}^r n_i\geq r+2=t+3$, which is a contradiction, since $H$ has $t$ vertices.
\end{proof}

If $t=1$, $H$ is an isolated vertex, and hence $G_2$, in the previous corollary, is just a cone. So, we obtain \cite[Lemma~3.4 and Corollary~3.7]{RR}:

\begin{cor}\label{one connected-cone}
Let $H_1,\ldots,H_r$ be connected graphs on $[n_1],\ldots,[n_r]$, respectively,
where $r\geq 2$, and assume that $G=\mathrm{cone}(v,\bigsqcup_{i=1}^r H_i)$.
Then $J_{G}$ is unmixed if and only if $J_{H_i}$ is unmixed for $i=1,\ldots,r$, and $r=2$.
\end{cor}

The following theorem is on the Cohen-Macaulay property:

\begin{thm}\label{one connected-CM}
Let $G_1,H_1,\ldots,H_r$ be connected graphs on $[t],[n_1],\ldots,[n_r]$, respectively,
where $r\geq 2$. Suppose that $G_2:=\bigsqcup_{i=1}^r H_i$ and $S/J_{G_2}$ is Cohen-Macaulay.
Then $S/J_{G_1*G_2}$ is Cohen-Macaulay if and only if $r=2$ and $t=1$.
\end{thm}

\begin{proof}
If $r=2$ and $t=1$, then the result follows by \cite[Theorem~3.8]{RR}. Now, let $G:=G_1*G_2$ and $n=t+\sum_{i=1}^rn_i$, and assume that $S/J_G$ is Cohen-Macaulay. So, $J_G$ is unmixed, and hence $r=t+1$ and $G_1$ is complete, by Corollary~\ref{one connected-m=2}. By Remark~\ref{G_1 complete},
$\mathcal{C}(G)=\{\emptyset\}\cup \{[t]\cup (\bigcup_{i=1}^r T_i):T_i\in \mathcal{C}(H_i),\mathrm{for}~i=1,\ldots,r\}$. So, we have $J_G=Q\cap Q'$, where
\begin{equation}
Q=\bigcap_{\substack{
T\in \mathcal{C}(G) \\
[t]\nsubseteq T
}}
P_T(G)=P_{\emptyset}(G)
\nonumber
\end{equation}
and
\begin{equation}
Q'=\bigcap_{\substack{
T\in \mathcal{C}(G) \\
[t]\subseteq T
}}
P_T(G).
\nonumber
\end{equation}
Now, consider the following short exact sequence
$$0\rightarrow S/J_G\rightarrow S/Q\oplus S/Q'\rightarrow S/(Q+Q')\rightarrow 0.$$
Because $Q=P_{\emptyset}(G)$ is a determinantal ideal, we have $\mathrm{depth}\hspace{0.35mm}S/Q=n+1$.
For every $\emptyset\neq T \in \mathcal{C}(G)$, we have $[t]\subseteq T$ and $P_T(G)=(x_i,y_i~:~i\in [t])+P_{T\setminus [t]}(G_2)$. Thus, $Q'=(x_i,y_i~:~i\in [t])+J_{G_2}$. Let $S_1=K\big{[}x_i,y_i:i\notin [t]\big{]}$. Then, $S/Q'\cong S_1/J_{G_2}$ is Cohen-Macaulay, and hence $\mathrm{depth}\hspace{0.35mm}S/Q'=\mathrm{dim}\hspace{0.35mm}S/Q'=(n-t)+r$, by \cite[Corollary~3.4]{HHHKR}. So, $\mathrm{depth}\hspace{0.35mm}S/Q'=n+1$, since $r=t+1$. Thus, $\mathrm{depth}\hspace{0.35mm}(S/Q\oplus S/Q')=n+1$.
One has $Q+Q'=(x_i,y_i~:~i\in [t])+P_{\emptyset}(G)$. So that $S/(Q+Q')\cong S_1/J_{\widetilde{G_2}}$, where $\widetilde{G_2}$ is the complete graph on
$[n]\setminus [t]=\bigcup_{i=1}^{r}[n_i]$. Therefore, $\mathrm{depth}\hspace{0.35mm}S_1/(Q+Q')=\mathrm{depth}\hspace{0.35mm}S_1/J_{\widetilde{G_2}}=(n-t)+1$. So, $\mathrm{depth}\hspace{0.35mm}S_1/(Q+Q')=n-r+2$,
since $r=t+1$. So, by using the depth lemma, we have $\mathrm{depth}\hspace{0.35mm}S/J_{G}=n-r+3$, because $r\geq 2$. On the other hand,
$S/J_{G_2}$ is Cohen-Macaulay, and hence $\mathrm{dim}\hspace{0.35mm}S/J_{G}=n+1$, by Proposition~\ref{one connected} and \cite[Corollary~3.4]{HHHKR}. But,
$S/J_{G}$ is Cohen-Macaulay, so $\mathrm{depth}\hspace{0.35mm}S/J_{G}=n+1$, and hence $r=2$ and $t=1$, as desired.
\end{proof}

We denote the path over $n$ vertices, by $P_n$. In addition, by $G^c$, we mean the complementary graph of the graph $G$.

\begin{exam}\label{fan}
{\em (a) Let $r,t\geq 1$. Suppose that $F_{r,t}$ is a \textit{fan graph}, which is $K_r^c*P_t$. By Proposition~\ref{one connected}, we have
$\mathrm{dim}\hspace{0.35mm}S/J_{F_{r,t}}=\mathrm{max}\{2r,r+t+1\}$, since $J_{P_t}$ is unmixed and hence $\mathrm{dim}\hspace{0.35mm}S/P_T(P_t)=t+1$, for all $\emptyset \neq T \in \mathcal{C}(P_t)$. Moreover, by Theorem~\ref{one connected-unmixed2}, $P_3$ and $F_{3,2}$ (which is also a generalized block graph) are the only fan graphs whose binomial edge ideals are unmixed. Also, the only fan graph, whose binomial edge ideal is Cohen-Macaulay, is $F_{2,1}=P_3$.

(b) Let $l,s\geq 2$, and $G:=K_{l}^c*K_{s-1}$, the complete $s$-partite graph whose $s-1$ parts consist of only one vertex, and one part consists of $l$ vertices. By Proposition~\ref{one connected}, we have $\mathrm{dim}\hspace{0.35mm}S/J_{G}=\mathrm{max}\{2l,l+s\}$. Also, by Theorem~\ref{one connected-unmixed2}, $J_{G}$ is unmixed if and only if $l=s$. In this case, we get $\mathrm{dim}\hspace{0.35mm}S/J_{G}=2s$.  }
\end{exam}

\subsection{Join of two disconnected graphs}\label{Join of two disconnected graphs}

Now, we consider the join of two disconnected graphs.

\begin{prop}\label{both disconnected 2}
Suppose that $G_1=\bigsqcup_{i=1}^r G_{1i}$ and $G_2=\bigsqcup_{i=1}^s G_{2i}$ are two graphs on disjoint sets of vertices
$[n_1]=\bigcup_{i=1}^r [n_{1i}]$ and $[n_2]=\bigcup_{i=1}^s [n_{2i}]$,
respectively, where $r,s\geq 2$. Then we have \\\\
{\em{(a)}} $\mathcal{C}(G_1*G_2)=\{\emptyset\}\cup \big{(}(\bigcirc_{i=1}^{r}\mathcal{C}(G_{1i}))\circ \{[n_2]\} \big{)}\cup \big{(}(\bigcirc_{i=1}^{s}\mathcal{C}(G_{2i}))\circ \{[n_1]\} \big{)}.$ \\\\
{\em{(b)}} $\mathrm{height}\hspace{0.35mm}J_{G_1*G_2}=\mathrm{min}\{2n_2+\sum_{i=1}^r
\mathrm{height}\hspace{0.35mm}J_{G_{1i}},2n_1+\sum_{j=1}^s
\mathrm{height}\hspace{0.35mm}J_{G_{2j}},n_1+n_2-1\}$. \\\\
{\em{(c)}} $\mathrm{dim}\hspace{0.35mm}S/J_{G_1*G_2}=\mathrm{max}\{\sum_{i=1}^r \mathrm{dim}\hspace{0.35mm}S_{1i}/J_{G_{1i}},
\sum_{i=1}^s \mathrm{dim}\hspace{0.35mm}S_{2i}/J_{G_{2i}},n_1+n_2+1\}$,
where $S_{1i}=K\big{[}x_j,y_j:j\in [n_{1i}]\big{]}$ for $i=1,\ldots,r$, and $S_{2i}=K\big{[}x_j,y_j:j\in [n_{2i}]\big{]}$ for $i=1,\ldots,s$.
\end{prop}

\begin{proof}
\indent (a) Let $G:=G_1*G_2$ and $T\in (\bigcirc_{i=1}^{r}\mathcal{C}(G_{1i}))\circ \{[n_2]\}$. So, $T=[n_2]\cup (\bigcup_{i=1}^r T_{1i})$, where $T_{1i}\in \mathcal{C}(G_{1i})$, for $i=1,\ldots,r$. We show that $T$ has cut point property. Let $j\in T$. If $j\in T_{1i}$, for some $i=1,\ldots,r$, then $G_{([n]\setminus T)\cup \{j\}}={G_{1i}}_{([n_{1i}]\setminus T_{1i})\cup \{j\}}\sqcup (\bigsqcup_{l=1,l\neq i}^r {G_{1l}}_{([n_{1l}]\setminus T_{1l})})$. In this case, $j$ is a cut point of ${G_{1i}}_{([n_{1i}]\setminus T_{1i})\cup \{j\}}$, since $T_{1i}\in \mathcal{C}(G_{1i})$. So that $j$ is also a cut point of $G_{([n]\setminus T)\cup \{j\}}$. If $j\in [n_2]$, then $G_{([n]\setminus T)\cup \{j\}}=j*\bigsqcup_{i=1}^r {G_{1i}}_{([n_{1i}]\setminus T_{1i})}$. So, $j$ is a cut point of $G_{([n]\setminus T)\cup \{j\}}$, since $G_{([n]\setminus T)}$ is disconnected. Thus, in both cases, $T$ has cut point property. If $T\in (\bigcirc_{i=1}^{s}\mathcal{C}(G_{2i}))\circ \{[n_1]\}$, then similarly, we have $T\in \mathcal{C}(G)$.
For the other inclusion, let $\emptyset \neq T\in \mathcal{C}(G)$. If $T$ does not contain $[n_1]$ and $[n_2]$, then $G_{[n]\setminus T}$ is connected, and hence no element $i$ of $T$ is a cut point of $G_{([n]\setminus T)\cup \{i\}}$. So, we have $[n_1]\subseteq T$ or $[n_2]\subseteq T$. Suppose that $[n_1]\subseteq T$. Then, $T=[n_1]\cup (\bigcup_{i=1}^s T_{2i})$, where $T_{2i}\subseteq [n_2]$, for $i=1,\ldots,s$. Let $1\leq i\leq s$. If $T_{2i}=\emptyset$, then, clearly, $T_{2i}\in \mathcal{C}(G_{2i})$. If $T_{2i}\neq \emptyset$, then
each $j\in T_{2i}$, is a cut point of $G_{([n]\setminus T)\cup \{j\}}$, since $T\in \mathcal{C}(G)$. So that $j$ is a cut point of ${G_{2i}}_{([n_{2i}]\setminus T_{2i})\cup \{j\}}$, because $j\in T_{2i}$ and $G_{2i}$'s are on disjoint sets of vertices. Thus, $T_{2i}\in \mathcal{C}(G_{2i})$. Therefore, $T\in (\bigcirc_{i=1}^{s}\mathcal{C}(G_{2i}))\circ \{[n_1]\}$. If $[n_2]\subseteq T$, then similarly we get $T\in (\bigcirc_{i=1}^{r}\mathcal{C}(G_{1i}))\circ \{[n_2]\}$. \\
\indent Using part~(a), parts (b) and (c) follow by similar discussions as in the proof of Proposition~\ref{both connected}.
\end{proof}

The following corollary generalizes \cite[Theorem~1.1, part~(a)]{SZ}, due to Schenzel and Zafar, for complete $t$-partite graphs. Here, we denote by $K_{n_1,\ldots,n_t}$, a complete $t$-partite graph with parts of $n_1,\ldots,n_t$ vertices.

\begin{cor}\label{Schenzel-Zafar}
Let $1\leq n_1\leq \cdots \leq n_t$ be some integers. Then we have $$\mathrm{dim}\hspace{0.35mm}S/J_{K_{n_1,\ldots,n_t}}=\mathrm{max}\big{\{}\sum_{i=1}^tn_i+1,2n_t\big{\}}.$$
\end{cor}

\begin{proof}
We use induction on $t$, the number of parts of a complete $t$-partite graph. If $t=2$, then we have a complete bipartite graph $K_{n_1,n_2}$, with $n_1\leq n_2$. If $n_1=n_2=1$, the result is obvious. If $n_1=1$ and $n_2\geq 2$, then we have $\mathrm{dim}\hspace{0.35mm}S/J_{K_{n_1,n_2}}=2(n-1)$, which implies the result in this case. If $2\leq n_1\leq n_2$, then, by Proposition~\ref{both disconnected 2}, we have $\mathrm{dim}\hspace{0.35mm}S/J_{K_{n_1,n_2}}=\mathrm{max}\{2n_1,2n_2,n_1+n_2+1\}=\mathrm{max}\{2n_2,n_1+n_2+1\}$, which yields the result in this case. Now, let $t>2$ and $G=K_{n_1,\ldots,n_t}$ be the complete $t$-partite graph with $1\leq n_1\leq \cdots \leq n_t$. If $n_1=\cdots=n_t=1$, then $G$ is complete and hence the result is obvious. Now, suppose that $n_t\geq 2$. Then $G$ is the join of a complete $(t-1)$-partite graph on $\bigcup_{i=1}^{t-1}[n_i]$, say $G_1$, and $n_t$ isolated vertices. Thus, by Proposition~\ref{one connected}, we have $\mathrm{dim}\hspace{0.35mm}S/J_G=\mathrm{max}\{2n_t,\mathrm{dim}\hspace{0.35mm}S'/P_T(G_1),\sum_{i=1}^tn_i+1: T\in \overline{\mathcal{C}}(G_1)\}$, where $S'=K\big{[}x_i,y_i:i\in \bigcup_{i=1}^{t-1}[n_i]\big{]}$. On the other hand, $\mathrm{dim}\hspace{0.35mm}S'/P_T(G_1)\leq \mathrm{dim}\hspace{0.35mm}S'/J_{G_1}=\mathrm{max}\{\sum_{i=1}^{t-1}n_i+1,2n_{t-1}\}\leq \mathrm{max}\{\sum_{i=1}^{t}n_i+1,2n_t\}$, by the induction hypothesis. Thus, we have $\mathrm{dim}\hspace{0.35mm}S/J_G=\mathrm{max}\{2n_t,\sum_{i=1}^tn_i+1\}$, as desired.
\end{proof}

\begin{cor}\label{both disconnected-unmixed}
Let $G_1$ be a connected graph on $m\geq 2$ vertices, and assume that $H_1=\bigsqcup_{i=1}^r H_{1i}$ and $H_2=\bigsqcup_{i=1}^s H_{2i}$ are two graphs on
$[n_1]=\bigcup_{i=1}^r [n_{1i}]$ and $[n_2]=\bigcup_{i=1}^s [n_{2i}]$,
respectively, where $r,s\geq 2$. If $G_2:=H_1*H_2$, and $n_{1i},n_{2j}\geq m$ for all $i=1,\ldots,r$ and $j=1,\ldots,s$, then $J_{G_1,G_2}$ is
never unmixed, and hence Cohen-Macaulay.
\end{cor}

\begin{proof}
Let $T_1=[n_1]$ and $T_2=[n_2]$. By Proposition~\ref{both disconnected 2}, $T_1,T_2\in \mathcal{C}(G)$, and also $c_{G_2}(T_1)=s$ and $c_{G_2}(T_2)=r$.
Suppose on the contrary that $J_{G_1,G_2}$ is unmixed. Thus, we have $(s-1)(m-1)=n_1$ and $(r-1)(m-1)=n_2$, by Proposition~\ref{unmixed3}. But it is a contradiction, since $n_{1i},n_{2j}\geq m$, for all $i=1,\ldots,r$ and $j=1,\ldots,s$.
\end{proof}

With the notation of Corollary~\ref{both disconnected-unmixed}, if there are some $n_{1i}$ or $n_{2j}$ which are less than $m$,
then $J_{G_1,G_2}$ might be unmixed or not. For instance, \cite[Proposition~2.3 and Theorem~2.4]{O1}, provide some examples for this purpose.
Also, see the next example. We denote the $n$-cycle by $C_n$.

\begin{center}
\begin{figure}
\hspace{0 cm}
\includegraphics[height=2.7cm,width=2.5cm]{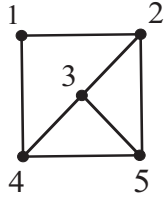}
\caption{\footnotesize{}The join of two graphs $K_2\sqcup K_1$ and $K_1\sqcup K_1$}\hspace{2 cm}
\label{graph}
\end{figure}
\end{center}

\begin{exam}\label{1}
\em{ (a) Let $G_1$ be the complete graph $K_m$. If $H_1$ and $H_2$ are both two isolated vertices, i.e. $K_1\sqcup K_1$,
then we have $n_{11}=n_{12}=n_{21}=n_{22}=1$, and $G_2=H_1*H_2$ is the $4$-cycle, $C_4$. Thus, by \cite[Proposition~4.2]{EHHQ}, $J_{G_1,G_2}$ is
unmixed if and only if $m=3$. \\
\indent (b) Let $G_1$ be the complete graph $K_m$, for some $m\geq 2$. If $H_1=K_2\sqcup K_1$ and $H_2=K_1\sqcup K_1$, then $n_{11}=2$, $n_{12}=n_{21}=n_{22}=1<m$, and $G_2=H_1*H_2$ is the graph in Figure~\ref{graph}. We have that
$J_{G_1,G_2}$ is never unmixed, since otherwise, for $T_1=\{1,3,5\}\in \mathcal{C}(G_2)$, we have $c_{G_2}(T_1)=2$ and hence $m=4$,
by Proposition~\ref{unmixed3}, but
for $T_2=\{2,4\}\in \mathcal{C}(G_2)$, we have $c_{G_2}(T_2)=2$ and hence $m=3$, again by Proposition~\ref{unmixed3}, which is a contradiction.}
\end{exam}

\section{ Binomial edge ideal of the corona of graphs }\label{Corona}

\noindent The \textit{corona product} $H\odot G$ of two graphs $H$ and $G$ is defined as the graph
obtained from $H$ and $G$ by taking one copy of $H$ and $|V(H)|$ copies of $G$ and joining
by an edge each vertex from the $i$-th copy of $G$ with the $i$-th vertex of $H$. For each
$v\in V(H)$, we often refer to $G_{v}$ for the copy of $G$ connected to $v$ in $H\odot G$. For example, Figure~\ref{corona1} shows the graph $K_3\odot K_2$ and Figure~\ref{corona2} depicts the graph $K_2\odot K_3$.

\begin{center}
\begin{figure}
\hspace{0 cm}
\includegraphics[height=3cm,width=3.5cm]{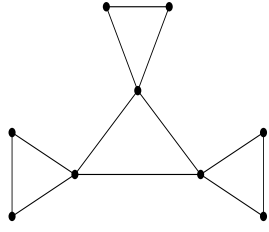}
\caption{\footnotesize{}The graph $K_3\odot K_2$}\hspace{2 cm}
\label{corona1}
\end{figure}
\end{center}

\begin{center}
\begin{figure}
\hspace{0 cm}
\includegraphics[height=2.2cm,width=3.6cm]{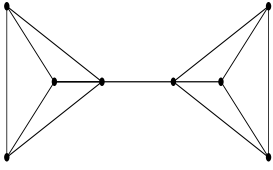}
\caption{\footnotesize{}The graph $K_2\odot K_3$}\hspace{2 cm}
\label{corona2}
\end{figure}
\end{center}

As a special case, when $G$ is just a vertex, $H\odot G$ is exactly $W(H)$, i.e. the graph which is obtained from a graph $H$ by adding a whisker to each of its vertices, whose some properties of the (monomial) edge ideal have been studied. For example, it is an interesting fact that the (monomial) edge ideal of a graph, obtained by adding whiskers to each vertex, is always Cohen-Macaulay. But, one could deduce from Theorem~\ref{CM-corona} that the binomial edge ideal of such a graph is Cohen-Macaulay if and only if the original graph is complete. In the sequel, we use the following, for simplicity:

Suppose that $H$ and $H'$ are connected graphs on disjoint sets of vertices $[n_1]$ and $[n_2]$, respectively. We consider the graph $H\odot H'$. Let $\emptyset \neq T\subseteq [n_1]$ and for all $v\in T$, $T_{v}\in \mathcal{C}(H'_{v})$ such that if $N_H(v)\subseteq T$, then $T_{v}\neq\emptyset$. Then, we say that $T, T_{v}$, for all $v\in T$, satisfy the property $\mathcal{P}$. Here, we mean by $N_H(v)$, the set of the neighbors of the vertex $v$ in $H$, i.e. vertices which are adjacent with $v$. Also, by ${[n_2]}_v$, we mean the set of vertices of the graph $H'_{v}$. If there is no confusion, we might write simply $[n_2]$ instead.

\begin{prop}\label{Prime}
If $H$ and $H'$ are connected graphs on disjoint sets of vertices $[n_1]$ and $[n_2]$, respectively, then we have \\\\
{\em{(a)}} $\mathcal{C}(H\odot H')=\{\emptyset\}\cup \{T\cup (\bigcup_{v\in T}T_{v}):T,T_{v},~\mathrm{for~all}~v\in T,~\mathrm{satisfy~the~property}~\mathcal{P} \}$. \\\\
{\em{(b)}} $\mathrm{height}\hspace{0.35mm}J_{H\odot H'}=\mathrm{min}\{\mathrm{height}\hspace{0.35mm}P_{T}(H)+
\sum_{v\in T}\mathrm{height}\hspace{0.35mm}P_{T_{v}}(H'_{v})+n_2(n_1-|T|),n_1+n_1n_2-1: T,T_{v},~\mathrm{for~all}~ v\in T,~\mathrm{satisfy~the~property}~\mathcal{P}\}$. \\\\
{\em{(c)}} $\mathrm{dim}\hspace{0.35mm}S/J_{H\odot H'}=\mathrm{max}\{\mathrm{dim}\hspace{0.35mm}S_1/P_{T}(H)+\sum_{v\in T}\mathrm{dim}\hspace{0.35mm}S_2/P_{T_{v}}(H'_{v})+n_2(n_1-|T|),n_1+n_1n_2+1:
T,T_{v},~\mathrm{for~all}~v\in T,~\mathrm{satisfy~the~property}~\mathcal{P}\}$, where $S_1=K\big{[}x_i,y_i:i\in [n_1]\big{]}$ and $S_2=K\big{[}x_i,y_i:i\in [n_2]\big{]}$.
\end{prop}

\begin{proof}
(a) Suppose that $G:=H\odot H'$ and $n:=n_1+n_1n_2$. Let $T':=T\cup (\bigcup_{v\in T}T_{v})$, where $T, T_{v}$, for all $v\in T$, satisfy the property $\mathcal{P}$. So that $\emptyset \neq T\subseteq [n_1]$ and for all $v\in T$, $T_{v}\in \mathcal{C}(H'_{v})$ such that if $N_H(v)\subseteq T$, then $T_{v}\neq\emptyset$. Let $w\in T'$. We show that $w$ is a cut point of $G_{([n]\setminus T')\cup \{w\}}$. If $w\in T$ and $N_H(w)\subseteq T$, then by the assumption, $\emptyset \neq T_w\in \mathcal{C}(H'_{w})$ and hence $w$ is a cut point of $w*{(H'_{w})}_{[n_2]\setminus T_w}$, since ${(H'_{w})}_{[n_2]\setminus T_w}$ is a disconnected graph. So that $w$ is a cut point of $G_{([n]\setminus T')\cup \{w\}}$, because the only connected component of $G_{([n]\setminus T')\cup \{w\}}$ which contains $w$ is $w*{(H'_{w})}_{[n_2]\setminus T_w}$. If $w\in T$ and $N_H(w)\nsubseteq T$, then deleting the vertex $w$ from $G_{([n]\setminus T')\cup \{w\}}$, gives at least two connected components, one containing a neighbor of $w$ in $H$, and one the graph ${(H'_{w})}_{[n_2]\setminus T_w}$. Thus, $w$ is also a cut point of $G_{([n]\setminus T')\cup \{w\}}$ in this case. If $w\in T_v$, for some $v\in T$, then $w$ is a cut point of ${(H'_{v})}_{({[n_2]}\setminus T_v)\cup \{w\}}$, because $T_v\in \mathcal{C}(H'_{v})$. Since the connected components of ${(H'_{v})}_{({[n_2]}\setminus T_v)\cup \{w\}}$ are also some of the connected components of $G_{([n]\setminus T')\cup \{w\}}$, so that $w$ is also a cut point of $G_{([n]\setminus T')\cup \{w\}}$. So that $w$ is a cut point of $G_{([n]\setminus T')\cup \{w\}}$. Conversely, if $\emptyset \neq T'\in \mathcal{C}(G)$, then we have $T'=T\cup (\bigcup_{v\in [n_1]}T_{v})$, for some $T\subseteq [n_1]$ and $T_v\subseteq {[n_2]}_v$, for all $v\in [n_1]$. If $T_v\neq \emptyset$, for some $v\in [n_1]$, then $v\in T$, since otherwise, no vertex $w$ of $T_v$ is a cut point of $G_{([n]\setminus T')\cup \{w\}}$, a contradiction. So, since $T'\neq \emptyset$, we have $T\neq \emptyset$. Thus, we have $T'=T\cup (\bigcup_{v\in T}T_{v})$ for some $\emptyset \neq T\subseteq [n_1]$ and $T_v\subseteq {[n_2]}_v$, for all $v\in T$. Let $v\in T$ and $w\in T_v$. Since $T'\in \mathcal{C}(G)$, we have that $w$ is a cut point of $G_{([n]\setminus T')\cup \{w\}}$. Thus, obviously, $w$ is also a cut point of ${(H'_{v})}_{({[n_2]}_v\setminus T_v)\cup \{w\}}$, which implies that $T_v\in \mathcal{C}(H'_{v})$. Finally, let $v\in T$ be such that $N_H(v)\subseteq T$ and $T_v=\emptyset$. Then $v$ is not a cut point of $G_{([n]\setminus T')\cup \{v\}}$, since it is adjacent to all the vertices of the connected graph ${(H'_{v})}_{{[n_2]}_v}$. It is a contradiction, because $T'\in \mathcal{C}(G)$. Thus, if $N_H(v)\subseteq T$, then $T_{v}\neq\emptyset$. \\
\indent Using (a), parts (b) and (c) follow by similar discussions as in the proof of Proposition~\ref{both connected}.
\end{proof}

\begin{cor}\label{H:complete}
Let $H$ and $H'$ be connected graphs on disjoint sets of vertices $[n_1]$ and $[n_2]$, respectively. If $H$ is complete and $J_{H'}$ is unmixed, then $$\mathrm{dim}\hspace{0.35mm}S/J_{H\odot H'}=n_1+n_1n_2+1.$$
\end{cor}

\begin{proof}
Let $\emptyset \neq T\subseteq [n_1]$ and for all $v\in T$, $T_v\in \mathcal{C}(H')$ with the property $\mathcal{P}$. We have $\mathrm{dim}\hspace{0.35mm}S_1/P_{T}(H)+\sum_{v\in T}\mathrm{dim}\hspace{0.35mm}S_2/P_{T_{v}}(H'_{v})+n_2(n_1-|T|)=\big{(}n_1-|T|+c_{H}(T)\big{)}+\big{(}\sum_{v\in T}(n_2-|T_v|+c_{H'}(T_v))\big{)}+n_2(n_1-|T|)$. Since $H$ is complete, we have $c_H(T)=1$, and since $J_{H'}$ is unmixed, we have $c_{H'}(T_v)=|T_v|+1$, by Proposition~\ref{unmixed1}. Thus, we get $\mathrm{dim}\hspace{0.35mm}S_1/P_{T}(H)+\sum_{v\in T}\mathrm{dim}\hspace{0.35mm}S_2/P_{T_{v}}(H'_{v})+n_2(n_1-|T|)=n_1n_2+n_1+1$. Hence, by Proposition~\ref{Prime}, part~(c), we get the result.
\end{proof}

\begin{rem}
{\em In Corollary~\ref{H:complete}, if $H$ is  not complete or $J_{H'}$ is not unmixed, then the dimension of $S/J_{H\odot H'}$ might not be equal to $n_1+n_1n_2+1$, but equal to the other term appeared in Proposition~\ref{Prime}, part~(c). For instance, if $H=H'=P_3$, then one could easily check that $\mathrm{dim}\hspace{0.35mm}S/J_{H\odot H'}=14$, but $n_1+n_1n_2+1=13$. Also, if $H=K_2$ and $H'=K_{1,3}$, then one could check that $J_{H'}$ is not unmixed and $\mathrm{dim}\hspace{0.35mm}S/J_{H\odot H'}=12$, but $n_1+n_1n_2+1=11$. }
\end{rem}

Let $H$ and $H'$ be connected graphs on disjoint sets of vertices $[n_1]$ and $[n_2]$, respectively. If $n_1=n_2=1$, then $H\odot H'$ is just an edge. If $n_1=1$ and $n_2\geq 2$, then $H\odot H'$ is the cone of a vertex on the graph $H'$, which was studied in Section~\ref{Join} and also in \cite{RR}. Now, we focus on the other cases. The following theorem determines those graphs whose binomial edge ideal of their corona is Cohen-Macaulay and unmixed.

\begin{thm}\label{CM-corona}
Let $H$ and $H'$ be connected graphs on disjoint sets of vertices $[n_1]$ and $[n_2]$, respectively, where $n_1\geq 2$, $n_2\geq 1$.
Then the following conditions are equivalent: \\
{\em{(a)}} $S/J_{H\odot H'}$ is Cohen-Macaulay. \\
{\em{(b)}} $J_{H\odot H'}$ is unmixed. \\
{\em{(c)}} $H$ and $H'$ are complete.
\end{thm}

\begin{proof}
(a) $\Rightarrow$ (b) is well-known. \\
\indent (b) $\Rightarrow$ (c) follows by Lemma~\ref{Whisker1}, Lemma~\ref{unmixed1-corona} and Lemma~\ref{unmixed2-corona}, which will be proved in the sequel. Indeed, suppose that $J_{H\odot H'}$ is unmixed. If $n_2=1$, then the result is clear from Lemma~\ref{Whisker1}. If $n_2\geq 2$, then by Lemma~\ref{unmixed1-corona}, $J_{H'}$ is unmixed, since with the notation of Lemma~\ref{unmixed1-corona}, we have $m=2$ and $G=K_2$. Then, by Lemma~\ref{unmixed2-corona}, unmixedness of $J_{H'}$ and $J_{H\odot H'}$ yields that $H$ and $H'$ are complete graphs.  \\
\indent (c) $\Rightarrow$ (a) Note that $K_{n_1}\odot K_{n_2}$ could be seen as the graph obtained from gluing each vertex of $K_{n_1}$ to a vertex of a copy of $K_{n_2+1}$. Then the result follows by \cite[Theorem~2.7]{RR} and the fact that $S/J_{K_{n_1}}$ and $S/J_{K_{n_2+1}}$ are Cohen-Macaulay.
\end{proof}

Now, we prove the following lemmas which were applied in the proof of Theorem~\ref{CM-corona}.

\begin{lem}\label{Whisker1}
Let $G$ and $H$ be connected graphs on disjoint sets of vertices $[m]$ and $[n]$, respectively, where $2\leq m\leq n$.
Then the following conditions are equivalent: \\
{\em{(a)}} $S/J_{G,W(H)}$ is Cohen-Macaulay. \\
{\em{(b)}} $J_{G,W(H)}$ is unmixed. \\
{\em{(c)}} $G$ and $H$ are complete and $m=2$.
\end{lem}

\begin{proof}
(a) $\Rightarrow$ (b) is well-known. \\
\indent (b) $\Rightarrow$ (c) Suppose that $J_{G,W(H)}$ is unmixed. Then, by Proposition~\ref{unmixed3}, $G$ is a complete graph. Now, we show that $H$ is complete. Suppose on the contrary that $H$ is not complete. Thus, there exists $\emptyset \neq T\in \mathcal{C}(H)$. By Proposition~\ref{Prime}, part~(a), we have $T\in \mathcal{C}(W(H))$. Since $J_{G,W(H)}$ is unmixed, we have $(c_{W(H)}(T)-1)(m-1)=|T|$, again by Proposition~\ref{unmixed3}. One can easily see that $c_{W(H)}(T)=c_H(T)+|T|$. Thus, we have $(c_H(T)+|T|-1)(m-1)=|T|$, which is a contradiction, because $m\geq 2$ and $c_H(T)\geq 2$. So that $H$ is a complete graph. Now, let $\emptyset \neq T\in \mathcal{C}(W(H))$. So, $\emptyset \neq T\subsetneq [n]$, by Proposition~\ref{Prime}, part~(a). Thus, we have $c_{W(H)}(T)=c_{H}(T)+|T|$. But, $c_{H}(T)=1$, since $H$ is complete. Hence, we have $c_{W(H)}(T)=|T|+1$. Since $J_{G,W(H)}$ is unmixed, we have $(c_{W(H)}(T)-1)(m-1)=|T|$, by Proposition~\ref{unmixed3}. Thus, $|T|(m-1)=|T|$, which implies that $m=2$. \\
\indent (c) $\Rightarrow$ (a) It suffices to apply \cite[Theorem~2.7]{RR}, and the fact that $S/J_{K_{n}}$ and $S/J_{K_{2}}$ are Cohen-Macaulay.
\end{proof}

\begin{lem}\label{unmixed1-corona}
Let $G$, $H$ and $H'$ be connected graphs on disjoint sets of vertices $[m]$, $[n_1]$ and $[n_2]$, respectively, where $2\leq m\leq n_1,n_2$.
If $J_{G,H\odot H'}$ is unmixed, then \\
{\em{(a)}} $H$ is complete, and \\
{\em{(b)}} $J_{G,H'}$ is unmixed if and only if $m=2$.
\end{lem}

\begin{proof}
(a) It is enough to repeat the discussion in the proof of Lemma~\ref{Whisker1} \big{(}(b) $\Rightarrow$ (c)\big{)}, and
substitute $H\odot H'$ instead of $W(H)$. \\
\indent (b) Since $J_{G,H\odot H'}$ is unmixed, we have $G$ is complete, by Proposition~\ref{unmixed3}. Let $\emptyset \neq T\in \mathcal{C}(H')$. Let $v\in [n_1]$ and put $T_v:=T\in \mathcal{C}(H'_{v})$, that is the corresponding vertices of $T$ in the copy $H'_{v}$ of $H'$. Let $T':=T_v\cup \{v\}$. Then, we have $T'\in \mathcal{C}(H\odot H')$, by Proposition~\ref{Prime}. So, we have $|T'|=|T|+1$ and $c_{H\odot H'}(T')=c_{H}(\{v\})+c_{H'}(T)=c_{H'}(T)+1$, as $H$ is complete and hence $c_{H}(\{v\})=1$. Thus, we have $(c_{H'}(T)-1)(m-1)=(c_{H\odot H'}(T')-2)(m-1)=(c_{H\odot H'}(T')-1)(m-1)-(m-1)=|T'|-m+1=|T|-m+2$, where the third equality occurs by using Proposition~\ref{unmixed3}, because $J_{G,H\odot H'}$ is unmixed. Hence, by Proposition~\ref{unmixed3}, we have $J_{G,H'}$ is unmixed if and only if $m=2$, as desired.
\end{proof}

\begin{lem}\label{unmixed2-corona}
Let $G$, $H$ and $H'$ be connected graphs on disjoint sets of vertices $[m]$, $[n_1]$ and $[n_2]$, respectively, where $2\leq m\leq n_1,n_2$.
If $J_{G,H'}$ is unmixed, then the following conditions are equivalent: \\
{\em{(a)}} $J_{G,H\odot H'}$ is unmixed. \\
{\em{(b)}} $H$ and $H'$ are complete, and $m=2$.
\end{lem}

\begin{proof}
If $J_{G,H\odot H'}$ is unmixed, then, by Lemma~\ref{unmixed1-corona}, $H$ is complete. Now, let $\emptyset \neq T'\in \mathcal{C}(H\odot H')$. Then, by Proposition~\ref{Prime}, we have $T'=T\cup(\bigcup_{v\in T}T_v)$, where $\emptyset \neq T\subseteq [n_1]$ and $T_v\in \mathcal{C}(H'_{v})$ such that $T, T_v$ satisfy the property $\mathcal{P}$. So, $|T'|=|T|+\sum_{v\in T}|T_v|$ and $c_{H\odot H'}(T')=c_{H}(T)+\sum_{v\in T}c_{H'}(T_v)$. Now, consider two following cases: \\

(i) Suppose that $T=[n_1]$. Then, for all $v\in T$, $N_H(v)\subseteq T$, and hence $T_v\neq \emptyset$, by Proposition~\ref{Prime}, part~(a). Thus, in this case $H'$ is not complete. Moreover, we have $c_{H}(T)=0$ and so $c_{H\odot H'}(T')=\sum_{v\in T}c_{H'}(T_v)$. Thus,
\begin{eqnarray}
\big{(}c_{H\odot H'}(T')-1\big{)}(m-1)&=& \big{(}\sum_{v\in T}c_{H'}(T_v)\big{)}(m-1)-(m-1)
\nonumber\\
\nonumber &= & {}
\sum_{v\in T}\big{(}c_{H'}(T_v)-1\big{)}(m-1)+|T|(m-1)-(m-1)
\nonumber\\
\nonumber &=& {} \sum_{v\in T}|T_v|+|T|(m-1)-(m-1)
\nonumber\\
\nonumber &= & {}\big{(}|T'|-|T|\big{)}+|T|(m-1)-(m-1)
\nonumber\\
\nonumber &= & {}|T'|+|T|(m-2)-(m-1),
\nonumber
\end{eqnarray}
where the third equality holds by Proposition~\ref{unmixed3}, because $J_{G,H'}$ is unmixed. So that in this case, we get $\big{(}c_{H\odot H'}(T')-1\big{)}(m-1)=|T'|+|T|(m-2)-(m-1)$. \\

(ii) Suppose that $T\subsetneq [n_1]$. Then, $c_{H}(T)=1$, and hence $c_{H\odot H'}(T')=1+\sum_{v\in T}c_{H'}(T_v)$. Thus, similar to the previous case, we have
\begin{eqnarray}
\big{(}c_{H\odot H'}(T')-1\big{)}(m-1)&=& \big{(}\sum_{v\in T}c_{H'}(T_v)\big{)}(m-1)
\nonumber\\
\nonumber &= & {}
\sum_{v\in T}\big{(}c_{H'}(T_v)-1\big{)}(m-1)+|T|(m-1)
\nonumber\\
\nonumber &=& {} \sum_{v\in T}|T_v|+|T|(m-1)
\nonumber\\
\nonumber &= & {}\big{(}|T'|-|T|\big{)}+|T|(m-1)
\nonumber\\
\nonumber &= & {}|T'|+|T|(m-2).
\nonumber
\end{eqnarray}
So, in this case, we have $\big{(}c_{H\odot H'}(T')-1\big{)}(m-1)=|T'|+|T|(m-2)$. \\

Now, if $J_{G,H\odot H'}$ is unmixed and $H'$ is not complete, then case~(i) might happen. Thus, by Proposition~\ref{unmixed3}, we have $|T'|+|T|(m-2)-(m-1)=|T'|$, which yields that $|T|=n_1=\frac{m-1}{m-2}$. Hence, $m=3$ and $n_1=2$, which contradicts the assumption $m\leq n_1$. Thus, we get that if $J_{G,H\odot H'}$ is unmixed, then $H'$ is complete. Now, assuming $H$ and $H'$ are both complete, we show that $J_{G,H\odot H'}$ is unmixed if and only if $m=2$. Since we supposed that $H'$ is complete, case~(i) does not occur. So, by case~(ii) and using Proposition~\ref{unmixed3}, we have $J_{G,H\odot H'}$ is unmixed if and only if $|T'|=|T'|+|T|(m-2)$ if and only if $m=2$, since $|T|\geq 1$. Thus, we get the desired result.
\end{proof}

\textbf{Acknowledgments:} The research of the first author was in part supported by a grant from IPM (No. 93050220). The second author was supported by the German Research Council DFG-GRK~1916.

\providecommand{\byame}{\leavevmode\hbox
to3em{\hrulefill}\thinspace}

\end{document}